\documentclass[12pt]{amsart}
\usepackage{amsmath}
\usepackage{amssymb,float}
\usepackage[colorlinks=true]{hyperref}
\usepackage[all]{xy}
\usepackage[toc,page,title,titletoc,header]{appendix}
\usepackage{xcolor}
\usepackage{tikz-cd}
\usepackage{graphicx}
\usepackage{epigraph}
\usepackage{bbm}

\begin{document}
	\pdfoutput=1
	\theoremstyle{plain}
	\newtheorem{thm}{Theorem}[section]
	\newtheorem*{thm1}{Theorem 1}
	
	\newtheorem*{thmM}{Main Theorem}
	\newtheorem*{thmA}{Theorem A}
	\newtheorem*{thm2}{Theorem 2}
	\newtheorem{lemma}[thm]{Lemma}
	\newtheorem{lem}[thm]{Lemma}
	\newtheorem{cor}[thm]{Corollary}
	\newtheorem{pro}[thm]{Proposition}
	\newtheorem{prop}[thm]{Proposition}
	\newtheorem{variant}[thm]{Variant}
	\theoremstyle{definition}
	\newtheorem{notations}[thm]{Notations}
	\newtheorem{rem}[thm]{Remark}
	\newtheorem{rmk}[thm]{Remark}
	\newtheorem{rmks}[thm]{Remarks}
	\newtheorem{defi}[thm]{Definition}
	\newtheorem{exe}[thm]{Example}
	\newtheorem{claim}[thm]{Claim}
	\newtheorem{ass}[thm]{Assumption}
	\newtheorem{prodefi}[thm]{Proposition-Definition}
	\newtheorem{que}[thm]{Question}
	\newtheorem{con}[thm]{Conjecture}
	
	\newtheorem{exa}[thm]{Example}
	\newtheorem*{assa}{Assumption A}
	\newtheorem*{algstate}{Algebraic form of Theorem \ref{thmstattrainv}}
	
	\newtheorem*{dmlcon}{Dynamical Mordell-Lang Conjecture}
	\newtheorem*{condml}{Dynamical Mordell-Lang Conjecture}
	\newtheorem*{congb}{Geometric Bogomolov Conjecture}
	\newtheorem*{congdaocurve}{Dynamical Andr\'e-Oort Conjecture for curves}
	
	\newtheorem*{pdd}{P(d)}
	\newtheorem*{bfd}{BF(d)}

	\newtheorem*{probreal}{Realization problems}
	\numberwithin{equation}{section}
	\newcounter{elno}                
	\def\points{\list
		{\hss\llap{\upshape{(\roman{elno})}}}{\usecounter{elno}}}
	\let\endpoints=\endlist
	\newcommand{\SH}{\rm SH}
	\newcommand{\Cov}{\rm Cov}
	\newcommand{\Tan}{\rm Tan}
	\newcommand{\res}{\rm res}
	\newcommand{\Om}{\Omega}
	\newcommand{\om}{\omega}
	\newcommand{\La}{\Lambda}
	\newcommand{\la}{\lambda}
	\newcommand{\mc}{\mathcal}
	\newcommand{\mb}{\mathbb}
	\newcommand{\surj}{\twoheadrightarrow}
	\newcommand{\inj}{\hookrightarrow}
	\newcommand{\zar}{{\rm zar}}
	\newcommand{\Exc}{{\rm Exc}}
	\newcommand{\Mod}{{\rm Mod}}
	\newcommand{\an}{{\rm an}}
	\newcommand{\red}{{\rm red}}
	\newcommand{\codim}{{\rm codim}}
	\newcommand{\Supp}{{\rm Supp\;}}
	\newcommand{\Leb}{{\rm Leb}}
	\newcommand{\rank}{{\rm rank}}
	\newcommand{\Lip}{{\rm Lip}}
	\newcommand{\Ker}{{\rm Ker \ }}
	\newcommand{\Pic}{{\rm Pic}}
	\newcommand{\Der}{{\rm Der}}
	\newcommand{\Div}{{\rm Div}}
	\newcommand{\Hom}{{\rm Hom}}
	\newcommand{\Corr}{{\rm Corr}}
	\newcommand{\im}{{\rm im}}
	\newcommand{\Spec}{{\rm Spec \,}}
	\newcommand{\Nef}{{\rm Nef \,}}
	\newcommand{\Frac}{{\rm Frac \,}}
	\newcommand{\Sing}{{\rm Sing}}
	\newcommand{\sing}{{\rm sing}}
	\newcommand{\reg}{{\rm reg}}
	\newcommand{\Char}{{\rm char\,}}
	\newcommand{\Tr}{{\rm Tr}}
	\newcommand{\ord}{{\rm ord}}
	\newcommand{\bif}{{\rm bif}}
	\newcommand{\AS}{{\rm AS}}
	\newcommand{\FS}{{\rm FS}}
	\newcommand{\CE}{{\rm CE}}
	\newcommand{\PCE}{{\rm PCE}}
	\newcommand{\WR}{{\rm WR}}
	\newcommand{\PR}{{\rm PR}}
	\newcommand{\TCE}{{\rm TCE}}
	\newcommand{\diam}{{\rm diam\,}}
	\newcommand{\id}{{\rm id}}
	\newcommand{\NE}{{\rm NE}}
	\newcommand{\Gal}{{\rm Gal}}
	\newcommand{\Min}{{\rm Min \ }}
	\newcommand{\Hol}{{\rm Hol \ }}
	\newcommand{\Rat}{{\rm Rat}}
	\newcommand{\Hdim}{{\rm Hdim}}
	\newcommand{\dist}{{\rm dist}}
	\newcommand{\FL}{{\rm FL}}
	\newcommand{\fm}{{\rm fm}}
	\newcommand{\vol}{{\rm vol}}
	
	\newcommand{\Max}{{\rm Max \ }}
	\newcommand{\Alb}{{\rm Alb}\,}
	\newcommand{\Aff}{{\rm Aff}\,}
	\newcommand{\GL}{{\rm GL}\,}        
	\newcommand{\PGL}{{\rm PGL}\,}
	\newcommand{\Bir}{{\rm Bir}}
	\newcommand{\Bif}{{\rm Bif}}
	\newcommand{\Aut}{{\rm Aut}}
	\newcommand{\topo}{{\rm top}}
	\newcommand{\End}{{\rm End}}
	\newcommand{\Per}{{\rm Per}\,}
	\newcommand{\Preper}{{\rm Preper}\,}
	\newcommand{\Preim}{{\rm Preim}\,}
	\newcommand{\ie}{{\it i.e.\/},\ }
	\newcommand{\niso}{\not\cong}
	\newcommand{\nin}{\not\in}
	\newcommand{\soplus}[1]{\stackrel{#1}{\oplus}}
	\newcommand{\by}[1]{\stackrel{#1}{\rightarrow}}
	\newcommand{\longby}[1]{\stackrel{#1}{\longrightarrow}}
	\newcommand{\vlongby}[1]{\stackrel{#1}{\mbox{\large{$\longrightarrow$}}}}
	\newcommand{\ldownarrow}{\mbox{\Large{\Large{$\downarrow$}}}}
	\newcommand{\lsearrow}{\mbox{\Large{$\searrow$}}}
	\renewcommand{\d}{\stackrel{\mbox{\scriptsize{$\bullet$}}}{}}
	\newcommand{\dlog}{{\rm dlog}\,}    
	\newcommand{\longto}{\longrightarrow}
	\newcommand{\vlongto}{\mbox{{\Large{$\longto$}}}}
	\newcommand{\limdir}[1]{{\displaystyle{\mathop{\rm lim}_{\buildrel\longrightarrow\over{#1}}}}\,}
	\newcommand{\liminv}[1]{{\displaystyle{\mathop{\rm lim}_{\buildrel\longleftarrow\over{#1}}}}\,}
	\newcommand{\norm}[1]{\mbox{$\parallel{#1}\parallel$}}
	\newcommand{\boxtensor}{{\Box\kern-9.03pt\raise1.42pt\hbox{$\times$}}}
	\newcommand{\into}{\hookrightarrow}
	\newcommand{\image}{{\rm image}\,}
	\newcommand{\Lie}{{\rm Lie}\,}      
	\newcommand{\CM}{\rm CM}
	\newcommand{\Ma}{\mathbf{M}}
	\newcommand{\Teich}{\rm Teich\;}
	\newcommand{\genus}{{\rm genus}}
	\newcommand{\gonality}{{\rm gonal}}
	\newcommand{\sext}{\mbox{${\mathcal E}xt\,$}}  
	\newcommand{\shom}{\mbox{${\mathcal H}om\,$}}  
	\newcommand{\coker}{{\rm coker}\,}  
	\newcommand{\sm}{{\rm sm}}
	\newcommand{\pgcd}{\text{pgcd}}
	\newcommand{\trd}{\text{tr.d.}}
	\newcommand{\tensor}{\otimes}
	\newcommand{\hotimes}{\hat{\otimes}}
	
	\newcommand{\CH}{{\rm CH}}
	\newcommand{\tr}{{\rm tr}}
	\newcommand{\e}{\rm SH}
	
	\renewcommand{\iff}{\mbox{ $\Longleftrightarrow$ }}
	\newcommand{\supp}{{\rm supp}\,}
	\newcommand{\esssup}{{\rm ess\,sup}}
	\newcommand{\ext}[1]{\stackrel{#1}{\wedge}}
	\newcommand{\onto}{\mbox{$\,\>>>\hspace{-.5cm}\to\hspace{.15cm}$}}
	\newcommand{\propsubset}
	{\mbox{$\textstyle{
				\subseteq_{\kern-5pt\raise-1pt\hbox{\mbox{\tiny{$/$}}}}}$}}
	\newcommand{\sA}{{\mathcal A}}
	\newcommand{\sB}{{\mathcal B}}
	\newcommand{\sC}{{\mathcal C}}
	\newcommand{\sD}{{\mathcal D}}
	\newcommand{\sE}{{\mathcal E}}
	\newcommand{\sF}{{\mathcal F}}
	\newcommand{\sG}{{\mathcal G}}
	\newcommand{\sH}{{\mathcal H}}
	\newcommand{\sI}{{\mathcal I}}
	\newcommand{\sJ}{{\mathcal J}}
	\newcommand{\sK}{{\mathcal K}}
	\newcommand{\sL}{{\mathcal L}}
	\newcommand{\sM}{{\mathcal M}}
	\newcommand{\sN}{{\mathcal N}}
	\newcommand{\sO}{{\mathcal O}}
	\newcommand{\sP}{{\mathcal P}}
	\newcommand{\sQ}{{\mathcal Q}}
	\newcommand{\sR}{{\mathcal R}}
	\newcommand{\sS}{{\mathcal S}}
	\newcommand{\sT}{{\mathcal T}}
	\newcommand{\sU}{{\mathcal U}}
	\newcommand{\sV}{{\mathcal V}}
	\newcommand{\sW}{{\mathcal W}}
	\newcommand{\sX}{{\mathcal X}}
	\newcommand{\sY}{{\mathcal Y}}
	\newcommand{\sZ}{{\mathcal Z}}
	\newcommand{\A}{{\mathbb A}}
	\newcommand{\B}{{\mathbb B}}
	\newcommand{\C}{{\mathbb C}}
	\newcommand{\D}{{\mathbb D}}
	\newcommand{\E}{{\mathbb E}}
	\newcommand{\F}{{\mathbb F}}
	\newcommand{\G}{{\mathbb G}}
	\newcommand{\HH}{{\mathbb H}}
	\newcommand{\LL}{{\mathbb L}}
	\newcommand{\J}{{\mathbb J}}
	\newcommand{\M}{{\mathbb M}}
	\newcommand{\N}{{\mathbb N}}
	\renewcommand{\P}{{\mathbb P}}
	\newcommand{\Q}{{\mathbb Q}}
	\newcommand{\R}{{\mathbb R}}
	\newcommand{\T}{{\mathbb T}}
	\newcommand{\U}{{\mathbb U}}
	\newcommand{\V}{{\mathbb V}}
	\newcommand{\W}{{\mathbb W}}
	\newcommand{\X}{{\mathbb X}}
	\newcommand{\Y}{{\mathbb Y}}
	\newcommand{\Z}{{\mathbb Z}}
	\newcommand{\ch}{{\mathbbm {1}}}
	\newcommand{\bk}{{\mathbf{k}}}
	
	\newcommand{\bp}{{\mathbf{p}}}
	\newcommand{\ep}{\varepsilon}
	\newcommand{\bbk}{{\overline{\mathbf{k}}}}
	\newcommand{\Fix}{\mathrm{Fix}}
	
	\newcommand{\tor}{{\mathrm{tor}}}
	\renewcommand{\div}{{\mathrm{div}}}
	
	\newcommand{\trdeg}{{\mathrm{trdeg}}}
	\newcommand{\Stab}{{\mathrm{Stab}}}
	
	\newcommand{\OK}{{\overline{K}}}
	\newcommand{\ok}{{\overline{k}}}
	
	\newcommand{\cf}{{\color{red} [c.f. ?]}}
	\newcommand{\jy}{\color{red} jy:}

	\title[]{Equidistribution speed of  iterated preimages for rational maps on the Riemann sphere}

	\author{Mai Hao}

	
	\address{Institute for Theoretical Sciences, Westlake University, Hangzhou 310030, China}
	
	\email{haomai@westlake.edu.cn}

	\author{Zhuchao Ji}
	
	\address{Institute for Theoretical Sciences, Westlake University, Hangzhou 310030, China}
	
	\email{jizhuchao@westlake.edu.cn}

	
	\date{\today}

	\bibliographystyle{alpha}
	
	\maketitle
	
	\begin{abstract}
		The exponential equidistribution speed of iterated preimages for holomorphic endomorphisms on $\P^k$  was established by Drasin--Okuyama for  $k=1$, and by Dinh--Sibony for arbitrary $k$. 
		In this paper,  we obtain a near-optimal equidistribution speed  with order $O(nd^{-n})$ in dimension one  for points that are not super-attracting periodic. 
		Moreover, the equidistribution speed order $O(nd^{-n})$ holds not only for $\sC^2$ observables but also for  H\"older continuous  d.s.h. observables. 
		For geometrically finite rational maps (including all hyperbolic rational maps), we prove that the equidistribution speed order  is $O(d^{-n})$ for $\sC^2$ observables and points that are not super-attracting, attracting, or parabolic  periodic.
	\end{abstract}
	\tableofcontents

	\section{Introduction}
	
	Let $f$ be a rational map on the Riemann sphere $\P^1$ of degree $d\geq 2$.  
	The celebrated equidistribution theorem of Brolin \cite{MR194595}, Lyubich \cite{MR741393},  and Freire-Lopez-Ma\~n\'e \cite{MR736568}  asserts that the preimages $d^{-n}(f^n)^*(\delta_a)$ of $a\in\P^1$ converges in the weak-* sense  to the unique maximal entropy measure $\mu_f$, provided that $a$ is not in the exceptional set $\sE_f$.  Here $\sE_f$ contains all points that have a finite backward orbit. The cardinality of $\sE_f$ is at most 2. 
	The aim of this paper is to give an equidistribution speed for iterated preimages of $a$. 
	
	Drasin--Okuyama \cite{MR2346941} and Dinh--Sibony \cite{MR2640045} have established the following exponential equidistribution speed:

	\begin{thm}\label{thm:dods}
		Let $f$ be a rational map on $\P^1$ of degree $d\geq 2$.  Let $a\in \P^1$ which is not a super-attracting periodic point.  Then for any $1<\la<d$, there exists a constant $C=C(f,a,\la)>0$ such that for any $\sC^2$ observable $\phi:\P^1\to \R$, 
		$$|\langle d^{-n}(f^n)^*(\delta_a)-\mu_f, \phi\rangle|\leq C\la^{-n}||\phi||_{\sC^2}.$$
	\end{thm}
	
	To obtain equidistribution speed with order $\la^{-n}$ for any $1<\la<d$, the assumption that $a$  is not a  super-attracting periodic  point cannot be removed.   
	Dinh--Sibony \cite{MR2640045}  also obtained the exponential equidistribution speed of iterated preimages for holomorphic endomorphisms on $\P^k$ with $k\geq 1$.  
	The methods of Dinh--Sibony and Drasin--Okuyama are different.  Dinh--Sibony's method is more geometric, which is based on potential theory and estimate of the diameter of the pull-back of a ball.  Drasin--Okuyama's proof relies on Nevanlinna theory, which is more analytic.   
	
	\medskip

	In this paper we give a sharper equidistribution speed than Theorem \ref{thm:dods} not only for $\sC^2$ observables, but also for a class of {\bf H\"older continuous observables}, namely for H\"older continuous  d.s.h. observables.  A function on $\P^1$ is called {\em quasi-subharmonic} if it is locally the difference of a subharmonic function and a smooth function.   A function is called
	{\em d.s.h.} if it is equal outside a polar set to the difference of two quasi-subharmonic functions.  We identify two d.s.h. functions if they are equal out of a polar set.  If $\phi$ is a d.s.h. function, $dd^c \phi$ can be written as the difference of two finite positive measures with the same mass. Define $||dd^c\phi||$ to be the minimum of the mass of $\mu^++\mu^-$, where the minimum is taken over all finite positive measures $\mu^\pm$ such that $dd^c\phi=\mu^+-\mu^-$.
	
	Throughout the paper, the distances are computed with respect to the spherical metric,  normalized such that $\diam \P^1=1$.  
	For a rational map $f$, we define $\sS_f$ to be  the set containing  all  super-attracting periodic  point.
	
	\begin{thm}\label{thm:main}
		Let $f$ be a rational map on $\P^1$ of degree $d\geq 2$. Then  there exists a constant $C=C(f)>0$  such that for  $a\in \P^1$  not a super-attracting periodic  point, and for any $\alpha$-H\"older  d.s.h. observable $\phi:\P^1\to \R$, $0<\alpha\leq 1$, we have
		$$|\langle d^{-n}(f^n)^*(\delta_a)-\mu_f, \phi\rangle|\leq C\left[1+\log \frac{1}{d(a, \sS_f)}\right] nd^{-n}(||\phi||_\alpha+||dd^c\phi||).$$
		Here   $||\phi||_\alpha:=\sup_{x,y\in \P^1, x\neq y} \frac{|\phi(x)-\phi(y)|}{d(x,y)^\alpha}$ is the H\"older constant of $\phi$.
	\end{thm}
	
	\medskip
	
	When $\phi$ is a $\sC^2$ function, it is automatically a d.s.h. function. Since $||\phi||_1+||dd^c\phi||$ is bounded by the $\sC^2$ norm of $\phi$, we get:
	
	\begin{cor}\label{cor:general}
		Let $f$ be a rational map on $\P^1$ of degree $d\geq 2$. Then  there exists a constant $C=C(f)>0$  such that for  $a\in \P^1$  not a super-attracting periodic  point, for any $\sC^2$ observable $\phi:\P^1\to \R$, 
		$$|\langle d^{-n}(f^n)^*(\delta_a)-\mu_f, \phi\rangle|\leq C\left[1+\log \frac{1}{d(a, \sS_f)}\right] nd^{-n}||\phi||_{\sC^2}.$$	
	\end{cor}

	\medskip
	
	The proof of Theorem \ref{thm:main} is geometric, we estimate the diameter of the pull-back of a ball.
	
	Independently of our result,  for $\sC^2$ observables, recently Okuyama \cite{Okuyama} has obtained the same equidistribution speed $O(nd^{-n})$, but the dependence of the constant in his estimate on $a$  is inexplicit.  
	Okuyama's proof relies on Nevanlinna theory. 
	
	\medskip
	
	A rational map $f$ is called {\em geometrically finite} if  the critical points of $f$  on the Julia set $\sJ_f$ have finite orbits.  For geometrically finite maps,  we prove that the equidistribution speed is of order $O(d^{-n})$.   
	Every {\em hyperbolic} rational map (i.e. expanding on the Julia set) is  geometrically finite.
	
	\begin{thm}\label{thm:gf}
		Let $f$ be a geometrically finite rational map on $\P^1$ of degree $d\geq 2$.  
		Assume  $a\in \P^1$  is neither a super-attracting, attracting, nor parabolic periodic point, then there exists a constant $C=C(f,a)>0$  such that  for any $\sC^2$ observable $\phi:\P^1\to \R$, 
		$$|\langle d^{-n}(f^n)^*(\delta_a)-\mu_f, \phi\rangle|\leq C d^{-n}(||dd^c\phi||_\infty+||\phi||_\infty).$$
		Here as a signed measure $dd^c\phi=h\omega $, where $h$ is continuous and $\omega$ is the Fubini--Study form,  and we define $||dd^c\phi||_\infty:=||h||_\infty$.	
	\end{thm}

	Since $||dd^c\phi||_\infty+||\phi||_\infty$ is bounded by the $\sC^2$ norm of $\phi$, we get:
	
	\begin{cor}\label{cor:gf}
		Let $f$ be a geometrically finite  rational map on $\P^1$ of degree $d\geq 2$.  Assume  $a\in \P^1$  is neither an super-attracting, attracting, nor parabolic periodic point, then there exists a constant $C=C(f,a)>0$  such that  for any $\sC^2$ observable $\phi:\P^1\to \R$, 
		$$|\langle d^{-n}(f^n)^*(\delta_a)-\mu_f, \phi\rangle|\leq C d^{-n}||\phi||_{\sC^2}.$$
		
	\end{cor}
	
	The proof of Theorem \ref{thm:gf} is based on the distortion estimate of the inverse branches of $f^n$ on a small ball. 
	
	\medskip

	The equidistribution speed order $O(d^{-n})$ for H\"older continuous d.s.h. observables is the best we can expect in the following sense:
	
	\begin{thm}\label{thm:speed lower bound}
		Let $f$ be a rational map on $\P^1$ of degree $d\geq 2$ and $a\in \sJ_f$. Then there exist a H\"older continuous d.s.h. observable $\phi:\P^1\to \R$, an integer $m\geq 1$, a constant $C>0$ such that for every $n\geq 1$,
		$$|\langle d^{-mn}(f^{mn})^*(\delta_a)-\mu_f, \phi\rangle|\geq  C d^{-mn}.$$
	\end{thm}
	
	\medskip
	
	In Theorem \ref{thm:main} we show that the equidistribution speed  order has an  upper bound $O(nd^{-n})$. Compared to Theorem \ref{thm:speed lower bound},  it is natural to ask whether  $O(nd^{-n})$ is also a lower bound.  We propose the following question:
	
	\begin{que}\label{que:optimal}
		Let $d\geq 2$. Determine whether there exist a rational map on $\P^1$ of degree $d$,  a point $a\in \P^1$ which is not a super-attracting periodic point,  a H\"older continuous d.s.h. function $\phi:\P^1\to \R$, a constant $C>0$, and  a subsequence $n_j\to +\infty$ such that 
		$$|\langle d^{-n_j}(f^{n_j})^*(\delta_a)-\mu_f, \phi\rangle|\geq  C n_j d^{-n_j}.$$
	\end{que}

	\medskip

	Let us finally mention that the  exponential equidistribution speed of  periodic points for holomorphic endomorphisms on $\P^k,k\geq1$ was established by de Th\'elin--Dinh--Kaufmann \cite{de2024exponential}, see also Yap \cite{yap2024quantitative}, Dinh--Yap \cite{yap2025lower}, and Gauthier--Vigny \cite{gauthier2025quantitative}.
	
	\medskip

	\subsection*{Acknowledgement}
	The authors are supported by National Key R\&D Program
	of China (No.2025YFA1018300), NSFC Grant (No.12401106), and ZPNSF Grant (No.XHD24A0201).
	
	\medskip
	
	\section{Denker--Przytycki--Urba\'nski's lemma for Fatou critical points}
	Let $f$ be a rational map on $\P^1$ of degree $d\geq 2$ and denite by $\sC_f$ the critical set of $f$. The cardinality of $\sC_f$ is at most $2d-2$.  
	Let  $\sC'_f$ be the subset of $\sC_f$ removing periodic critical points and let $N$ be the cardinality of $\sC'_f$.   
	For $x\in \P^1$, $c\in\sC_f$,  we define the functions
	$$\psi_f(x):=-\log d(x,\sC'_f),$$
	$$\tilde{\psi}_f(x):=-\log d(x,\sC_f),$$
	and
	$$\psi_c(x):=-\log d(x,c).$$
	
	Since $\diam \P^1=1$,  $\psi_c(x)$,  $\psi_f(x)$ and $\tilde{\psi}_f(x)$ are  always non-negative.  The aim of this section is to prove the following result:
	\begin{prop}\label{prop:dpu}
		There is a constant $Q=Q(f)>0$ such that for any $x\in \P^1$ and $n\geq 1$, 
		$$\sum\psi_f(f^i(x))\leq Qn,$$
		where the summation is taken over all integers $0\leq i\leq n$ with at most $N$ terms removed. 
	\end{prop}
	\begin{proof}
		If we replace $\sC'_f$ by $\sC_f\cap \sJ_f$ in the definition of $\psi_f$, where $\sJ_f$ is the Julia set, the same statement of Proposition \ref{prop:dpu} was proved by Denker--Przytycki--Urba\'nski \cite[Lemma 2.3]{denker1996transfer}.  
		Thus we will get the conclusion once we prove the following lemma:
		
		\begin{lem}
			Let $c\in \sC'_f$ such that $c$ is contained in the Fatou set.  There exists a constant $Q>0$ such that for any $x\in \P^1$ and $n\geq 1$,  
			$$\sum\psi_c(f^i(x))\leq Qn$$
			where the summation is taken over all integers $0\leq i\leq n$ with at most one term removed. 
		\end{lem}
		\begin{proof}
			Since $c$ is not  periodic and is contained in the Fatou set,  by classification of Fatou components, $c$ is non-recurrent in the sense that $x\notin \omega(x)$, where $\omega(x)$ is the $\omega$-limit set of $x$. 
			Since $c$ is in the Fatou set, by normality argument,  there exists $\delta>0$ such that for the ball  $B:=B(c,\delta)$, we have $f^n(B)\cap B=\emptyset$ for every $n\geq 1$.  
			Set $Q:=-\log \delta$.  By our choice of $B$,  for the orbit $\left\{x,\dots, f^{n}(x)\right\}$,  there is at most one index $k$ such that $f^k(x)\in B$.  Thus 
			$$\sum\psi_c(f^i(x))\leq Qn,$$
			where the summation is taken over all integers in $\left\{0,\dots, n\right\}\setminus \left\{k\right\}$. This finishes the proof. 
		\end{proof}
		
	\end{proof}
	
	\medskip
	
	\section{Diameter of the pull-back of a ball}
	Let $f$ be a rational map on $\P^1$ of degree $d\geq 2$.   The following result was proved by Denker--Przytycki--Urba\'nski \cite[Lemma 3.4]{denker1996transfer}. 
	
	\begin{lem}\label{lem:dpudiameter}
		Fix $M>0$, $Q>0$. Then there exist constants  $L>0$, $\rho>0$  such that  the following hold:  if $x\in \P^1$ and $n\geq 1$ satisfy 
		$$\sum\tilde{\psi}_f(f^i(x))\leq Qn,$$
		where the summation is taken over all integers $0\leq i\leq n$ with at most  $M$ terms removed,  then for any ball $B$ such that $f^n(x)\in B$, let $V$ be the  connected component  of $f^{-n}(B)$ containing $x$, we have
		$$\diam V\leq L^n(\diam B)^\rho.$$
	\end{lem}

	\medskip
	
	Recall that $\sS_f$ is the set of super-attracting periodic points.  We fix a $\delta>0$ such that for every $y\in \sS_f$, the ball $B(y,\delta)$ is contained in the B\"ottcher coordinate of $y$. 
	Therefore, there exists a constant $C\geq 1$ such that  if the periods of $y$ is $m\geq 1$ and the multiplicity of $f^m$ at $y$ is $l'\geq 2$,  then for any set  $B\subset B(y,\delta)$ and $n\geq 1$,  we have
	\begin{equation}\label{eqn:Bottcher}
		\diam B\leq  C(\diam f^n(B))^{l^{-n}}
	\end{equation}here we set $l:=l'^{1/m}>1$. Let $B(\sS_f,\delta)$ be the $\delta$-neighborhood of $\sS_f$, which is forward invariant, i.e. $f(B(\sS_f,\delta))\subset B(\sS_f,\delta).$
	
	Combining Proposition \ref{prop:dpu} and Lemma \ref{lem:dpudiameter} we immediately get:
	
	\begin{lem}\label{lem:diameter}
		There exist $L>0$, $\rho>0$ such that for $x\in \P^1$ and $n\geq 1$ such that $f^n(x)\notin B(\sS_f,\delta)$, and for a ball $B$ such that $f^n(x)\in B$,  let $V$ be the connected component  of $f^{-n}(B)$ containing $x$, we have
		$$\diam V\leq L^n (\diam B)^\rho.$$
	\end{lem}
	\begin{proof}
		The forward invariance of $B(\sS_f,\delta)$ implies that $f^i(x)\notin B(\sS_f,\delta)$ for every $0\leq i\leq n$.  In particular, for every periodic critical point $c$, $-\log d(f^i(x),c)\geq -\log \delta$.  
		By Proposition \ref{prop:dpu}, there is a constant $Q=Q(f)>0$ such that 
		$$\sum \tilde{\psi}_f(f^i(x))\leq Qn,$$
		where the summation is taken over all integers $0\leq i\leq n$ with at most $N$ terms removed.  Applying Lemma \ref{lem:dpudiameter} and we get the result. 
	\end{proof}
	
	\medskip
	
	\begin{prop}\label{prop:diameter}
		There exist  $L>0$, $\rho>0$ such that  for $x\in \P^1$, $n\geq 1$ and a ball $B$ such that $f^n(x)\in B$, let $V$ be the connected component  of $f^{-n}(B)$ containing $x$, we have
		$$\diam V\leq L^n(\diam B)^\frac{\rho }{-\log d(f^n(x), \sS_f)+1}.$$
	\end{prop}
	\begin{proof}
		
		Let $l>1$ be the constant in (\ref{eqn:Bottcher}) which is required to be minimal along all $y\in \sS_f$.  
		There exists a constant $c=c(f)>0$ such that  if $l^{-m}(-\log d(z, \sS_f))<c$,  then for every $w\in f^{-m}(z)$, $w\notin B(\sS_f,\delta)$.  
		Applying this for $z=f^n(w)$ for some $w\in V$,  the condition
		\begin{equation}\label{eqn:3.5}
			m> \frac{\log (-\log d(z, \sS_f))-\log c }{\log l},
		\end{equation}
		implies that $f^{n-m}(w)\notin B(\sS_f,\delta)$. 
		
		There are two cases.
		
		{\bf Case 1:}  $V\subset  B(\sS_f,\delta)$, hence $f^i(V)\subset B(\sS_f,\delta)$ for $i\geq 0$.  Then by (\ref{eqn:Bottcher}), we have 
		$$\diam V\leq C(\diam B)^{l^{-n}}.$$
		
		By (\ref{eqn:3.5}), the condition $x\in B(\sS_f,\delta)$ implies that $$n\leq \frac{\log (-\log d(z, \sS_f))-\log c }{\log l}.$$
		
		Thus $$\diam V\leq C(\diam B)^{\frac{c}{-\log d(f^n(x), \sS_f)}}, $$
		which implies the desired result. 
		\medskip
		
		{\bf Case 2:} there exists 
		$0\leq m\leq n$  such that  $f^{n-m}(V)\not\subset B(\sS_f,\delta)$.  Let $0\leq m\leq n$ be the maximal integer that satisfies this property. 
		Pick $w\in V$ such that $f^{n-m}(w)\notin B(\sS_f,\delta)$.   If $m=0$, then $f^n(V)\not\subset B(\sS_f,\delta)$, we get the desired result by Lemma \ref{lem:diameter}.   
		So we can assume $1\leq m\leq n$.  The maximality of $m$ implies that $f^{n-m+1}(w)\in B(\sS_f,\delta)$.
		
		Thus  for $z:=f^n(w)$, by (\ref{eqn:3.5}) we have 
		\begin{equation}\label{eqn:3.3}
			m-1< \frac{\log (-\log d(z, \sS_f))-\log c }{\log l}.
		\end{equation}

		Let $V_{m}$ be the connected component  of $f^{-m}(B)$ containing $f^{n-m}(w)$.  By (\ref{eqn:Bottcher}) and (\ref{eqn:3.3}),  
		\begin{equation}\label{eqn:3.2}
			\diam V_m\leq  C(\diam B)^{l^{-m}}\leq C(\diam B)^{\frac{c'}{-\log d(f^n(x), \sS_f)}},
		\end{equation}
		where $c':=c/l$.
		Apply Lemma \ref{lem:diameter} to the orbit $\left\{w,\dots f^{n-m}(w)\right\}$, we get 
		\begin{equation}\label{eqn:3.4}
			\diam V\leq \tilde{L}^{n-m}(\diam V_m)^{\tilde{\rho}},
		\end{equation}
		where $\tilde{L} $ and $\tilde{\rho}$ are constants in Lemma \ref{lem:diameter}.  
		Let $L=\tilde{L}+C^{\tilde{\rho}}$ and $\rho=c\tilde{\rho}$, combining (\ref{eqn:3.4}) with (\ref{eqn:3.2}) we get the desired result.

	\end{proof}
	
	\medskip
	
	\section{Proof of Theorem \ref{thm:main}}\label{section:3}

	Let $f$ be a rational map on $\P^1$ of degree $d\geq 2$, and let $\mu_f$ be the unique maximal entropy measure.   For any $a\in \P^1$, Let $g_a$ be the unique d.s.h.  function such that $dd^c g_a=\delta_a-\mu_f$ and $\sup_{a\in \P^1}g_a=0$.  
	Since $\mu_f$ has continuous local potential \cite[Proposition 1.18]{dinh2010dynamics}, there exists $A=A(f)>0$  such that for any  $M>0$, the set $\left\{g_a<-M\right\}$ is contained in the ball $B(a, Ae^{-M})$.  
	Define $g_{a,M}:=\max\left\{g_a, -M\right\}.$ Let $\nu_{a,M}:=dd^c g_{a,M}+\mu_f$, which is a probability measure supported in $B(a, Ae^{-M})$.
	
	The following lemma is classical.
	\begin{lem}\label{lem:bdd}
		Let $\nu$ be a probability measure on $\P^1$ with bounded local potential,  $dd^c g=\nu-\mu_f.$  Then for any continuous d.s.h.  observable $\phi$ and $n\geq 1$, we have 
		$$|\langle d^{-n}(f^n)^*(\nu)-\mu_f, \phi\rangle|\leq d^{-n}  ||g||_\infty ||dd^c\phi||.$$
		
	\end{lem}
	\begin{proof}
		We have $d^{-n}(f^n)^*(\nu)-\mu_f=d^{-n} dd^c(g\circ f^n).$ Hence
		\begin{align*}
			|\langle d^{-n}(f^n)^*(\nu)-\mu_f, \phi\rangle|&=|\langle d^{-n} dd^c(g\circ f^n), \phi\rangle|
			\\&=|\langle  d^{-n} g\circ f^n, dd^c \phi\rangle|
			\\&\leq d^{-n} ||g||_\infty ||dd^c\phi|| .
		\end{align*}
	\end{proof}
	
	\proof[Proof of Theorem \ref{thm:main}]
	For each $n\geq 1$, we determine a number  $M_n$ such that  for the ball $B:=B(a, r_n)$, $r_n:=Ae^{-M_n}$, for any connected component $V$ of $f^{-n}(B)$, we have $\diam V\leq d^{-\frac{n}{\alpha}}$, where $0<\alpha\leq 1$ is the H\"older exponent of the observable $\phi$.
	By Proposition \ref{prop:diameter}, it suffices to choose $M_n$ such that 
	$$L^n r_n^\frac{\rho }{-\log d(a, \sS_f)+1}\leq d^{-\frac{n}{\alpha}}.$$
	
	A direct computation shows that we can choose $$M_n:=\beta(-\log d(a,\sS_f)+1)n+\log A,$$ whrere $\beta:=\frac{\log L+\log d-\log \alpha}{\rho}$, where $L,\rho$ are constants in Proposition \ref{prop:diameter}.
	
	Since the observable $\phi$ is $\alpha$-H\"older, 
	\begin{align}\label{eqn:4.1}
		|\langle d^{-n}(f^n)^*(\delta_a)-d^{-n}(f^n)^*(\nu_{a,M_n}), \phi\rangle|&\leq d^{-n} ||\phi||_\alpha \sum_{V} (\diam V)^\alpha  \notag
		\\&\leq d^{-n} ||\phi||_\alpha,
	\end{align}
	where the summation $\sum_{V} \diam V$ is taken over all connected components of $f^{-n}(B(a,r_n))$.  
	The first inequality holds since  both $(f^n)^*(\delta_a)$ and $(f^n)^*(\nu_{a,M_n})$ can be written as a summation of probability measures supported on  connected components  $V$ of $f^{-n}(B(a,r_n))$. 
	The second inequality holds since by our choice of $r_n$ we have $\diam V\leq d^{-\frac{n}{\alpha}}$ and there are at most $d^n$ terms in the summation.
	
	Recall $||g_{a,M_n}||_\infty =M_n$ and $dd^c g_{a,M_n}=\nu_{a,M_n}-\mu_f$.  By Lemma \ref{lem:bdd} we have for any continuous d.s.h. observable $\phi$, 
	\begin{align}\label{eqn:4.2}
		|\langle d^{-n}(f^n)^*(\nu_{a,M_n})-\mu_f, \phi\rangle|\leq d^{-n}||g_{a,M_n}||_\infty ||dd^c\phi||.
	\end{align}
	
	Combine (\ref{eqn:4.1}) and (\ref{eqn:4.2}),    for any H\"older continuous d.s.h. observable $\phi$ we have: 
	\begin{align*}
		|\langle d^{-n}(f^n)^*(\delta_a)-\mu_f, \phi\rangle|&\leq d^{-n}||\phi||_\alpha +M_nd^{-n}||dd^c\phi||)
		\\&\leq M_n d^{-n}(||\phi||_\alpha+||dd^c\phi||).
	\end{align*}
	
	By the definition of $M_n$ we get 
	$$|\langle d^{-n}(f^n)^*(\delta_a)-\mu_f, \phi\rangle|\leq C\left[1+\log\frac{1}{d(a, \sS_f)}\right] nd^{-n}(||\phi||_\alpha +||dd^c\phi||),$$
	by setting $C:=\beta+ \log A$. 
	\endproof
	
	\medskip
	
	\section{Proof of Theorem \ref{thm:gf}}
	
	Let $f$ be a geometrically finite rational map on $\P^1$ of degree $d\geq 2$.   Let $\sP(f):= \cup_{n\geq 1} f^n(\sC_f)$ be the post-critical set. 
	If $y$ is a $\omega$-limit point of $\sP(f)$, then $y$ is a super-attracting, attracting,  or parabolic periodic point. 
	
	
	\proof[Proof of Theorem \ref{thm:gf}]
	
	We first assume that $a\notin \overline{\sP_f}$.  Then there exists a ball $B'$ centered at $a$ such that for every $n\geq 1$ and every connected component $V$ of $f^{-n}(B')$, $f^n:V\to B'$ is injective. 
	Define $B:=\frac{1}{2}B'$.  By the Koebe distortion theorem, there is a uniform constant $\la>1$ such that  for every $n\geq 1$ and every connected component $V$ of $f^{-n}(B)$, for every $z,w\in V$
	\begin{equation}\label{eqn:koebe}
		\frac{1}{\la}\leq \frac{|df^n(z)|}{|df^n(w)|}\leq \la.
	\end{equation}

	Pick $M>0$ sufficiently large such that the support of the probability measure $\nu_{a,M}:=dd^c g_{a,M}+\mu_f$  is contained in $B$ (see the begining of Section 3 for the definition of $g_{a,M}$).   
	Both $(f^n)^*(\delta_a)$ and $(f^n)^*(\nu_{a,M})$ can be written as a summation of probability measures supported on  connected components  $V$ of $f^{-n}(B)$.  
	Fix one connected component $V$, and let $\nu$ be the branch of   $(f^n)^*(\nu_{a,M})$  on $V$ and let $\delta_y$ be the branch of  $(f^n)^*(\delta_a)$ on $V$. We have $f^n(y)=a$.  We need the following lemma:
	
	\begin{lem}\label{lem:distortion}
		There exists a  constant $\alpha=\alpha(f,a)>0$ such that for any $\sC^2$ observable $\phi:\P^1\to \R$,  
		$$|\langle \nu-\delta_y,  \phi\rangle|\leq\alpha  ||dd^c\phi||_\infty \vol (V).$$
	\end{lem}
	
	\begin{proof}
		Let $F:=f^n|_V$ be the restriction of $f^n$ on $V$.  Recall  $dd^c \phi= h\omega$, where $\omega$ is the Fubini--Study form, $h$ is continuous,  and $||dd^c \phi||_\infty:=||h||_\infty$.  We have
		\begin{align*}
			|\langle \nu-\delta_y,  \phi\rangle|&=|\langle \nu_{a,M}-\delta_a,  F_*\phi\rangle| 
			\\&=|\langle dd^c(g_{a,M}-g_a),  F_*\phi\rangle|
			\\&=|\langle g_{a,M}-g_a,  dd^c(F_*\phi)\rangle|
			\\&=|\langle g_{a,M}-g_a,  F_*(dd^c \phi)\rangle|
			\\&=|\langle g_{a,M}-g_a,  F_*(h\omega )\rangle|
			\\&\leq 2||h||_\infty  \langle g_{a,M}-g_a,   F_*(\omega)\rangle
			\\&\leq2 ||h||_\infty  \langle g_{a,M}-g_a,  \omega\rangle    \sup_{z\in V} \frac{1}{|dF(z)|^2}
			\\&\leq \alpha  ||h||_\infty \vol(V),
		\end{align*} 
		where $\alpha=\alpha(f,a)>0$ is a constant. 
		Let us explain why the last inequality holds. By (\ref{eqn:koebe}),  we have:   $$\vol(V)\geq  \eta \vol(B) \sup_{z\in V} \frac{1}{|dF(z)|^2}, $$ where $\eta>0$ is a uniform constant depending only on the metric on $\P^1$. 
		The integral  $\langle g_{a,M}-g_a,  \omega\rangle $ is also  bounded since $g_{a,M}-g_a$ is a difference of  d.s.h.  functions, which are Lebesgue integrable. 
	\end{proof}
	By Lemma \ref{lem:distortion}, for any $\sC^2$ observable $\phi$, 
	\begin{align}\label{eqn:5.2}
		|\langle d^{-n}(f^n)^*(\delta_a)-d^{-n}(f^n)^*(\nu_{a,M_n}), \phi\rangle|&=d^{-n}  \sum_{V} 	|\langle \nu-\delta_y,  \phi\rangle| \notag
		\\&\leq  \sum_{V}  \alpha  ||dd^c\phi||_\infty \vol(V) d^{-n} \notag 
		\\&\leq \beta ||dd^c\phi||_\infty d^{-n},
	\end{align}
	where the summation $\sum_{V} \diam V$ is taken over all connected components of $f^{-n}(B)$, and we set $\beta:=\alpha  \vol (\P^1)$. The last inequality holds since all the connected components of $f^{-n}(B)$ are disjoint, hence  $\sum_V \vol(V)\leq \vol(\P^1)$.

	By Lemma \ref{lem:bdd} we also have
	\begin{align}\label{eqn:5.3}
		|\langle d^{-n}(f^n)^*(\nu_{a,M})-\mu_f, \phi\rangle|\leq d^{-n}||g_{a,M}||_\infty ||dd^c\phi||\leq  M||dd^c\phi||_\infty  d^{-n}.
	\end{align}
	
	Combine (\ref{eqn:5.2}) and (\ref{eqn:5.3})  we have  for any $\sC^2$ observable $\phi$, 
	\begin{align*}
		|\langle d^{-n}(f^n)^*(\delta_a)-\mu_f, \phi\rangle|&\leq (M+\beta) d^{-n} ||dd^c\phi||_\infty,
	\end{align*}
	hence Theorem \ref{thm:gf} holds when $a\notin \overline{\sP_f}$.  
	
	\medskip
	
	Finally we consider the case $a\in \overline{\sP_f}$.  By our assumption that  $a$  is neither an super-attracting, attracting, nor parabolic periodic point, $a$  is not a $\omega$-limit point of $\sP(f)$. 
	Thus there exists an inter $N=N(f,a)>0$ such that if we let $y_1,\dots, y_{d^N}$ be all the points in $f^{-N}(a)$, counted with multiplicity,  then each $y_i\notin \overline{\sP_f}$. 
	Since Theorem \ref{thm:gf} holds when $y\notin \overline{\sP_f}$, for $n\geq N+1$, and for any $\sC^2$ observable $\phi$,  there is a constant $C=C(f,a)>0$ such that 
	\begin{align*}
		|\langle d^{-n}(f^n)^*(\delta_a)-\mu_f, \phi\rangle|&=\left|\Bigg\langle d^{-n}\left(\sum_{i=1}^{d^N}(f^{n-N})^*(\delta_{y_i})\right)-\mu_f, \phi\Bigg\rangle\right|
		\\&\leq C d^{-n} ||dd^c\phi||_\infty.
	\end{align*}
	
	For $1\leq n\leq N$, we have the trivial estimate
	\begin{align*}
		|\langle d^{-n}(f^n)^*(\delta_a)-\mu_f, \phi\rangle|\leq 2||\phi||_\infty.
	\end{align*}
	
	Thus  for $n\geq 1$, for any $\sC^2$ observable $\phi$,  there is a constant $C=C(f,a)>0$ such that 
	\begin{align*}
		|\langle d^{-n}(f^n)^*(\delta_a)-\mu_f, \phi\rangle|\leq C d^{-n} (||dd^c\phi||_\infty+||\phi||_\infty).
	\end{align*}

	Hence for $a\in \overline{\sP_f}$, the conclusion of Theorem \ref{thm:gf} still holds. This finishes the proof.

	\endproof
	
	\medskip
	
	\section{Proof of Theorem \ref{thm:speed lower bound}}
	
	\proof[Proof of Theorem \ref{thm:speed lower bound}]
	
	Recall that $g_a$ is  the unique d.s.h.  function such that $dd^c g_a=\delta_a-\mu_f$ and $\sup_{a\in \P^1}g_a=0$.  Then $-g_a$ is a positive function and $-g_a(x)\to+\infty$ when $x\to a$.  Let us construct  the  H\"older continuous d.s.h. function $\phi$ in Theorem \ref{thm:speed lower bound}.  Let $B_1, B_2$ be two  small balls such that $B_1$ is centered at $a$, $B_2\cap \sJ_f\neq \emptyset$,  moreover we ask
	$$A_1:=\inf_{x\in B_1}-g_a(z)>\sup_{x\in B_2}-g_a(x):=A_2.$$

	By the horseshoe construction, see for example \cite[Example 7.3]{ji2023homoclinic}, there exists two Cantor  expanding repellers $K_1\subset B_1$ and $K_2\subset B_2$, such that $f^{m_1}(K_1)=K_1$, $f^{m_2}(K_2)=K_2$, $m_i\geq 1$ are integers.  Both $f^{m_1}:K_1\to K_1$ and $f^{m_2}:K_2\to K_2$ have positive topological entropy, and their unique maximal entropy measure $\mu_i$, $i=1,2$, have H\"older continuous local potential.  Let $\phi$ be a d.s.h.  function such that $dd^c\phi=\mu_1-\mu_2$.  Then $\phi$ is H\"older continuous.   Set $m:=m_1m_2$.  Then we have
	\begin{align*}
		|\langle d^{-mn}(f^{mn})^*(\delta_a)-\mu_f, \phi\rangle|&=d^{-mn}|\langle dd^c (g_a\circ f^{mn}), \phi\rangle|
		\\&=d^{-mn}|\langle g_a\circ f^{mn}, dd^c \phi\rangle|
		\\&=d^{-mn}|\langle g_a\circ f^{mn}, \mu_1-\mu_2\rangle|
		\\&=d^{-mn}|\langle g_a,  \mu_1-\mu_2\rangle|
		\\&\geq d^{-mn}(A_1-A_2).
	\end{align*}
	
	Here  $\langle g_a\circ f^{mn}, \mu_1-\mu_2\rangle=\langle g_a,  \mu_1-\mu_2\rangle$ holds since both  $\mu_1$ and $\mu_2$ are $f^m$-invariant measures.  Since $A_1-A_2>0$, we finish the proof.
	\endproof

	

\end{document}